\documentclass{birkjour}

\usepackage[utf8]{inputenc}
\usepackage[T1]{fontenc}
\usepackage[english]{babel}
\usepackage{mathtools,amsmath,amssymb,amsfonts,amsthm,mathrsfs}
\mathtoolsset{showonlyrefs}

\usepackage{geometry}
\usepackage{graphicx,color}
\usepackage{tikz,pgfplots,pgf}
\usepackage[caption=false]{epsfig}
\usepackage{subcaption}
\usepackage[font=small]{caption}

\usepgfplotslibrary{polar}

\newtheorem{theorem}{Theorem}[section]

\newtheorem{proposition}{Proposition}[section]

\theoremstyle{definition}
\newtheorem{remark}{Remark}[section]

\numberwithin{equation}{section}

\begin{document}

\title[Periodic solutions to a perturbed relativistic Kepler problem]{Periodic solutions to a \\perturbed relativistic Kepler problem}

\author[A.~Boscaggin]{Alberto Boscaggin}

\address{
Department of Mathematics ``Giuseppe Peano'', University of Torino\\
Via Carlo Alberto 10, 10123 Torino, Italy}

\email{alberto.boscaggin@unito.it}

\author[W.~Dambrosio]{Walter Dambrosio}

\address{
Department of Mathematics ``Giuseppe Peano'', University of Torino\\
Via Carlo Alberto 10, 10123 Torino, Italy}

\email{walter.dambrosio@unito.it}

\author[G.~Feltrin]{Guglielmo Feltrin}

\address{
Department of Mathematics, Computer Science and Physics, University of Udine\\
Via delle Scienze 206, 33100 Udine, Italy}

\email{guglielmo.feltrin@uniud.it}

\thanks{Work written under the auspices of the Grup\-po Na\-zio\-na\-le per l'Anali\-si Ma\-te\-ma\-ti\-ca, la Pro\-ba\-bi\-li\-t\`{a} e le lo\-ro Appli\-ca\-zio\-ni (GNAMPA) of the Isti\-tu\-to Na\-zio\-na\-le di Al\-ta Ma\-te\-ma\-ti\-ca (INdAM). The third author is grateful to the other authors and to the University of Torino for the kind hospitality during the period in which the work was done.
\\
\textbf{Preprint -- March 2020}} 

\subjclass{34C25, 70H0B, 70H12, 83A05.}

\keywords{Relativistic Kepler problem, periodic solutions, invariant tori, nearly integrable Hamiltonian systems, action-angle coordinates.}

\date{}

\dedicatory{}

\begin{abstract}
We consider a perturbed relativistic Kepler problem
\begin{equation*}
\dfrac{\mathrm{d}}{\mathrm{d}t}\left(\dfrac{m\dot{x}}{\sqrt{1-|\dot{x}|^2/c^2}}\right)=-\alpha\, \dfrac{x}{|x|^3}+\varepsilon \, \nabla_x U(t,x), \qquad x \in \mathbb{R}^2 \setminus \{0\},
\end{equation*}
where $m, \alpha > 0$, $c$ is the speed of light and $U(t,x)$ is a function $T$-periodic in the first variable.
For $\varepsilon > 0$ sufficiently small, we prove the existence of $T$-periodic solutions with prescribed winding number, bifurcating from invariant tori of the unperturbed problem. 
\end{abstract}

\maketitle

\section{Introduction}\label{section-1}

The motion of a relativistic particle in a Kepler potential can be described by the equation
\begin{equation}\label{eq-0}
\dfrac{\mathrm{d}}{\mathrm{d}t}\left(\dfrac{m\dot{x}}{\sqrt{1-|\dot{x}|^2/c^2}}\right)=-\alpha\, \dfrac{x}{|x|^3}, \qquad x \in \mathbb{R}^2 \setminus \{0\},
\end{equation}
where the symbol $|\cdot|$ stands for the Euclidean norm of a two-dimensional vector, 
$m$ is the mass of the particle, $c$ is the speed of light and $\alpha$ is a constant ($m, \alpha > 0$).
Such an equation is well known in the physics community (see, for instance, \cite{AnBa-71,Bo-04,MuPa-06} and the references therein);
on the contrary, it has received much less attention by mathematicians working in the areas of Dynamical Systems and Nonlinear Analysis.
In particular, the dynamics of time-dependent perturbations of \eqref{eq-0} seems to be a quite unexplored topic and, to the best of our knowledge, only very few recent contributions can be quoted on this line of research \cite{Ga-19,ToUrZa-13,Za-13}.

In this paper, we deal with a time-periodic perturbation of the relativistic Kepler problem \eqref{eq-0}, precisely, 
\begin{equation}\label{eq-completa}
\dfrac{\mathrm{d}}{\mathrm{d}t}\left(\dfrac{m\dot{x}}{\sqrt{1-|\dot{x}|^2/c^2}}\right)=-\alpha\, \dfrac{x}{|x|^3}+\varepsilon \, \nabla_x U(t,x), \qquad x \in \mathbb{R}^2 \setminus \{0\},
\end{equation}
where $U \colon \mathbb{R} \times (\mathbb{R}^2 \setminus \{0\}) \to \mathbb{R}$ is a sufficiently regular function, $T$-periodic in the first variable for some $T>0$, and $\varepsilon > 0$ is a small parameter. In this context, a natural issue is the existence of $T$-periodic solutions of problem \eqref{eq-completa}, 
lying near the $T$-periodic solutions, if any, of the unperturbed problem ($\varepsilon = 0$).
Nearly circular solutions of \eqref{eq-completa} were recently provided in \cite[Corollary 9]{Ga-19}. 
As for solutions with a more complicated behavior, in this paper we prove, as a corollary of the main result
(Theorem~\ref{th-main}), the following theorem. 

\begin{theorem}\label{th-intro}
Let $U = U(t,x)\colon \mathbb{R} \times (\mathbb{R}^2 \setminus \{0\}) \to \mathbb{R}$ be a continuous function, continuously differentiable 
in the $x$-variable and $T$-periodic in the $t$-variable, for some
\begin{equation}\label{bound-T}
T > \dfrac{2\pi \alpha}{mc^3}.
\end{equation}
Then, for every sufficiently large integer $k \geq 2$, equation \eqref{eq-completa} has at least three $T$-periodic solutions with winding number equal to $k$ and three $T$-periodic solutions with winding number equal to $-k$, whenever $\varepsilon$ is sufficiently small. 
\end{theorem}

Notice that, as a consequence, equation \eqref{eq-completa} has as many $T$-periodic solutions as we wish, provided that $\varepsilon$ is small enough. On growing of the period $T$, even more $T$-periodic solutions could be obtained: we refer to the main result of the paper 
for the precise statement.

As it will be clear from the proof, \eqref{bound-T} is a necessary and sufficient condition for the existence of non-circular $T$-periodic solutions of the unperturbed problem \eqref{eq-0}. Let us notice that, whenever one of such solutions - say $x^*(t)$ - is found, a whole (two-parameter) family of $T$-periodic solutions actually exists: indeed, the functions $e^{i \omega}x^*(t + \tau)$, for
$\omega \in [0,2\pi)$ and $\tau \in [0,T)$, are also $T$-periodic solutions of \eqref{eq-0}.
In other terms, equation \eqref{eq-0} possesses an invariant two-dimensional torus filled by $T$-periodic solutions.

With this in mind, the strategy of the proof of Theorem~\ref{th-intro} can be easily illustrated.
First, we detect, for every positive real number $T$ satisfying \eqref{bound-T}, invariant tori filled by $T$-periodic solutions for the unperturbed equation: infinitely many of them exist, and they can be labelled by an integer $k$, giving the number of counter-clockwise rotations around the origin made in a period by the corresponding periodic solutions.
Second, we establish a bifurcation-type result from each of these invariant tori, ensuring the survival of a finite number of $T$-periodic solutions of the perturbed problem \eqref{eq-completa}, for $\varepsilon$ small enough. For this part of the proof, we rely on a recent result from \cite{FoGaGi-16}, providing - via an higher dimensional version of the Poincar\'e--Birkhoff fixed point theorem established in \cite{FoUr-17} - periodic solutions for nearly integrable Hamiltonian systems, whenever a suitable non-degeneracy condition 
for the unperturbed problem is satisfied. After finding action-angle coordinates for problem \eqref{eq-0}, we are thus led to check that such a non-degeneracy condition is satisfied, eventually concluding the proof of Theorem~\ref{th-intro}. We point out that, in general, the result is sharp: tori made by periodic solutions are typically destroyed by a perturbation, and no more than a finite number of periodic solutions can be found, for $\varepsilon \neq 0$. On the contrary, many quasi-periodic invariant tori will survive, by the methods of KAM theory 
(see \cite{Du-14} and Remark \ref{rem-kam}).

It is worth emphasizing that Theorem~\ref{th-intro} is in strong contrast with the corresponding results dealing with the perturbed Kepler problem 
\begin{equation}\label{eq-kep-intro}
m \ddot x =-\alpha\, \dfrac{x}{|x|^3}+\varepsilon\, \nabla_x U(t,x),
\end{equation}
arising by taking the non-relativistic limit $c \to +\infty$ in equation \eqref{eq-completa}. Indeed, for the Kepler problem the set of non-circular $T$-periodic solutions (that is, the set of Keplerian ellipses with positive eccentricity and fixed major semi-axis) is not a union of distinct two-dimensional tori, but rather a non-compact three-dimensional manifold, whose boundary includes circular solutions as well as collision-solutions. This is a consequence of the super-integrable character of the Kepler problem and eventually prevents the use of perturbation results as the one in \cite{FoGaGi-16}. Periodic solutions to \eqref{eq-kep-intro} are known to exist, for $\varepsilon \neq 0$, only in a generalized sense taking into account the possible occurrence of collisions \cite{BoDaPa-20,BoOrZh-19} or when the perturbation
term has some special structure (see \cite{AmCo-89,CaVi-00,FoGa-17,FoGa-18,SeTe-94} and the references therein).

\medskip

The plan of the paper is the following. In Section~\ref{section-2}, we deal with the unperturbed problem \eqref{eq-0},
with the final goals of detecting its periodic solutions, with explicit formulas for their minimal period, and finding action-angle coordinates.
In Section~\ref{section-3}, we first recall the bifurcation result from \cite{FoGaGi-16} and we then state and prove our main result.

\section{The unperturbed problem}\label{section-2}

In this section, we consider the unperturbed relativistic Kepler problem
\begin{equation}\label{eq-keplero}
\dfrac{\mathrm{d}}{\mathrm{d}t}\left(\dfrac{m\dot{x}}{\sqrt{1-|\dot{x}|^2/c^2}}\right)=-\alpha \, \dfrac{x}{|x|^3}, \qquad x \in \mathbb{R}^2 \setminus \{0\}.
\end{equation}
After writing it as a first order Hamiltonian system with two degrees of freedom (Section~\ref{subsec-hamilton}), 
we describe the set of its periodic solutions, providing an explicit formula for the minimal period (Section~\ref{subsec-periodiche} and Section~\ref{section-2.3}): in doing this, we partially follow and complete the analysis given in \cite{Bo-04,ToUrZa-13}. Lastly, we construct action-angle coordinates associated with \eqref{eq-keplero} (Section~\ref{section-2.4}).

\subsection{The Hamiltonian formulation}\label{subsec-hamilton}

In principle, equation \eqref{eq-keplero} appears as the Euler--Lagrange equation
\begin{equation*}
\dfrac{\mathrm{d}}{\mathrm{d}t} \frac{\partial E_0}{\partial\dot x} = \frac{\partial E_0}{\partial x}, 
\qquad \text{where }
E_0(x,\dot{x})=-mc^2\sqrt{1-\dfrac{|\dot{x}|^2}{c^2}}+\dfrac{\alpha}{|x|}.
\end{equation*}
For our purposes, however, it is more convenient to pass to a Hamiltonian formulation. This can be done (via Legendre transformation) defining the dual variable
\begin{equation}\label{eq-momentop}
p=\dfrac{m\dot{x}}{\sqrt{1-\dfrac{|\dot{x}|^2}{c^2}}}
\end{equation}
and the associated Hamiltonian
\begin{equation}\label{eq-hamiltoniana}
H_0(x,p)=mc^2\sqrt{1+\dfrac{|p|^2}{m^2c^2}}-\dfrac{\alpha}{|x|}.
\end{equation}
From now on, we thus consider the equivalent Hamiltonian system
\begin{equation*}
\dot x = \nabla_p H_0(x,p), \qquad
\dot p = - \nabla_x H_0(x,p),
\end{equation*}
which explicitly reads as
\begin{equation}\label{eq-hs}
\begin{cases}
\, \dot x = \dfrac{p}{m \sqrt{1 + \dfrac{| p |^2}{m^2c^2}}}, \vspace{7pt}\\
\, \dot p = - \alpha \, \dfrac{x}{|x|^3}.
\end{cases}
\end{equation}
Of course, the Hamiltonian $H_0$ is a first integral for \eqref{eq-hs}; moreover, it is plain to check that a second first integral is given by the angular momentum
\begin{equation*}
L_0(x,p) = \langle x, Jp \rangle, \qquad \text{where } J = \begin{pmatrix} 0 & 1 \\
-1 & 0\end{pmatrix}.
\end{equation*}

A first change of variables will be useful for investigating the dynamics of \eqref{eq-hs}. 
Precisely, we define
\begin{equation*}
\Omega=(0,+\infty)\times \mathbb{T}^1 \times \mathbb{R}^2
\end{equation*}
and the diffeomorphism
\begin{equation}\label{eq-cambiovar1}
\Psi \colon \Omega\to (\mathbb{R}^2\setminus\{0\})\times \mathbb{R}^2, \qquad (r,\vartheta,l,\Phi)\mapsto (x,p), 
\end{equation}
given by 
\begin{equation}\label{eq-cambiovar2}
x=re^{i\vartheta}, \qquad p=le^{i\vartheta}+\dfrac{\Phi}{r}ie^{i\vartheta}.
\end{equation}
Notice that 
\begin{equation*}
\langle x, Jp \rangle = \Phi,
\end{equation*}
showing that the variable $\Phi$ is nothing but the angular momentum. On the other hand, since
\begin{equation*}
\langle \dfrac{x}{|x|},p \rangle=l,
\end{equation*}
the variable $l$ will be called linear momentum.

This change of variables will be the first step in the construction of action-angles variables for the original problem. By now, we just observe that it is a symplectic map, meaning that
\begin{equation*}
\mathrm{d}x \wedge \mathrm{d}p = \mathrm{d}r \wedge \mathrm{d}l + \mathrm{d}\vartheta \wedge \mathrm{d}\Phi. 
\end{equation*}
As a consequence, system \eqref{eq-hs} is transformed into the Hamiltonian system
\begin{equation}\label{eq-hs2}
\left\{\begin{array}{lclcl}
\dot r &=& \partial_l \mathcal{H}_0(r,\vartheta,l,\Phi) &=& \dfrac{l}{m}\dfrac{1}{\sqrt{1+\dfrac{l^2+\Phi^2/r^2}{m^2c^2}}}, \vspace{7pt}\\
\dot l &=& - \partial_r \mathcal{H}_0(r,\vartheta,l,\Phi) &=& \dfrac{\Phi^{2}}{m r^{3}} \dfrac{1}{\sqrt{1+\dfrac{l^2+\Phi^2/r^2}{m^2c^2}}} - \dfrac{\alpha}{r^{2}}, \vspace{7pt}\\
\dot \vartheta &=& \partial_\Phi \mathcal{H}_0(r,\vartheta,l,\Phi) &=& \dfrac{\Phi}{mr^{2}}\dfrac{1}{\sqrt{1+\dfrac{l^2+\Phi^2/r^2}{m^2c^2}}}, \vspace{7pt}\\
\dot \Phi &=& -\partial_\vartheta \mathcal{H}_0(r,\vartheta,l,\Phi) &=& 0, \vspace{2pt}
\end{array}
\right.
\end{equation}
corresponding to the Hamiltonian
\begin{equation*}
\mathcal{H}_0(r,\vartheta,l,\Phi)=mc^2\sqrt{1+\dfrac{l^2+\Phi^2/r^2}{m^2c^2}}-\dfrac{\alpha}{r}.
\end{equation*}
As already seen, the angular momentum in the new variables writes as
\begin{equation*}
\mathcal{L}_0(r,\vartheta,l,\Phi)=\Phi.
\end{equation*}
Hence, every solution $(r,l,\vartheta,\Phi)=(r(t),l(t),\vartheta(t),\Phi(t))$ of \eqref{eq-hs2} satisfies
\begin{equation}\label{eq-traiettorie}
mc^2\sqrt{1+\dfrac{l(t)^2+L^2/(r(t))^2}{m^2c^2}}-\dfrac{\alpha}{r(t)}\equiv h,
\end{equation}
for some $h, L \in \mathbb{R}$.

\subsection{Dynamics in the $(r,l)$-plane} \label{subsec-periodiche}

In this section, we study the dynamics of the solutions of \eqref{eq-hs2} in the $(r,l)$-plane, with the aim of detecting values of energy $h$ and angular momentum $L \neq 0$ giving rise to closed orbits in such a plane.
Incidentally, let us notice that system \eqref{eq-hs2} cannot have periodic solutions with zero angular momentum (since $\dot l < 0$).

As already observed, for fixed values of the the energy $h\in \mathbb{R}$ and of the angular momentum $L\neq 0$, trajectories in the $(r,l)$-plane satisfy \eqref{eq-traiettorie}. Hence, we can obtain
\begin{equation}\label{eq-traiett2}
l^2=\phi_{h,L}(r),
\end{equation}
where
\begin{equation}\label{eq-funzphi}
\phi_{h,L}(r)=\dfrac{1}{c^2}\, \left(\dfrac{\alpha^2-L^2c^2}{r^2}+\dfrac{2\alpha h}{r}+h^2-m^2c^4\right), \quad r>0.
\end{equation}
The qualitative property of such a function of course changes on varying of the parameters $h,L$. 
We are going to look for situations where $\phi_{h,L}$ is positive between two consecutive zeros, both simple: clearly, this gives rise to a non-constant
closed orbit in the $(r,l)$-plane.

We first observe that
\begin{equation}\label{eq-zerophi}
\lim_{r\to 0^+} \phi_{h,L}(r)=
\begin{cases}
\,+\infty,&\text{if $\alpha^2-L^2c^2>0$, or if $\alpha^2-L^2c^2=0$, $h>0$,}\\
\,-\infty, &\text{if $\alpha^2-L^2c^2<0$, or if $\alpha^2-L^2c^2=0$, $h<0$,}\\
\, -m^2c^2, &\text{if $\alpha^2-L^2c^2=0$, $h=0$,}
\end{cases}
\end{equation}
and that
\begin{equation}\label{eq-infinitophi}
\lim_{r\to +\infty} \phi_{h,L}(r)=
\displaystyle \dfrac{h^2-m^2c^4}{c^2}.
\end{equation}
Moreover, it holds that
\begin{equation}\label{eq-derivataphi}
\phi'_{h,L}(r)=\dfrac{1}{c^2}\, \biggl{(}-\dfrac{2(\alpha^2-L^2c^2)}{r^3}-\dfrac{2\alpha h}{r^2}\biggr{)} 
=-\dfrac{2}{c^2r^3}\, (\alpha^2-L^2c^2+\alpha hr),
\end{equation}
for every $r > 0$.

We now distinguish several cases.

\smallskip
\noindent
\textit{Case~1: $\alpha^2-L^2c^2>0$.}
We consider two subcases:
\begin{itemize}
\item $h \geq 0$: in this case, from \eqref{eq-derivataphi} we immediately deduce that
$\phi'_{h,L}(r) < 0$ for every $r > 0$, so that there are no closed orbits.
\item $h < 0$: in this case, a simple computation shows that
$\phi'_{h,L}(r)< 0$ if and only if $r<r^{*}$,
where 
\begin{equation}\label{eq-rstarbasso}
r^{*}=\dfrac{\alpha^2-L^2c^2}{-\alpha h}>0.
\end{equation}
Hence, the function $\phi_{h,L}$ has a unique global minimum at $r = r^*$;
as a consequence, also in this case there are no closed orbits.
\end{itemize}

\smallskip
\noindent
\textit{Case~2: $\alpha^2-L^2c^2 = 0$.}
In this situation \eqref{eq-funzphi} reduces to
\begin{equation*}
\phi_{h,L}(r)=\dfrac{1}{c^2}\, \left(\dfrac{2\alpha h}{r}+h^2-m^2c^4\right),
\end{equation*}
so that the conclusion are straightforward. Precisely, we have:
\begin{itemize}
\item $h > 0$: in this case, $\phi_{h,L}$ is strictly decreasing and so there are no closed orbits.
\item $h = 0$: in this case, $\phi_{h,L}$ is a negative constant and no motion is possible.
\item $h < 0$: in this case, $\phi_{h,L}$ is strictly increasing and so there are no closed orbits.
\end{itemize}

\smallskip
\noindent
\textit{Case~3: $\alpha^2-L^2c^2 < 0$.} 
We consider two subcases:
\begin{itemize}
\item $h \leq 0$: in this case, from \eqref{eq-derivataphi} we immediately deduce that
$\phi'_{h,L}(r) > 0$ for every $r > 0$, so that there are no closed orbits.
\item $h > 0$: in this case, a simple computation shows that
$\phi'_{h,L}(r)> 0$ if and only if $r<r^{*}$ (with $r^*$ defined in \eqref{eq-rstarbasso}), so that
the function $\phi_{h,L}$ has a unique global maximum at $r = r^*$, 
with
\begin{equation*}
\phi_{h,L}(r^{*})=\dfrac{1}{c^2(L^2c^2-\alpha^2)} \bigl{(} \alpha^2 h^2+(h^2-m^2c^4)(L^2c^2-\alpha^2) \bigr{)}.
\end{equation*}
Recalling \eqref{eq-zerophi} and \eqref{eq-infinitophi}, we thus deduce that a first necessary condition for 
the existence of two simple zeros is that $h^2 - m^2 c^4 < 0$. At this point, the second condition
to be imposed is that $\phi_{h,L}(r^*) > 0$, giving rise to 
\begin{equation*}
L^2 < \dfrac{\alpha^2m^2c^2}{m^2c^4-h^2}.
\end{equation*}
\end{itemize}

Summarizing the above discussion, the orbits in the $(r,l)$-plane are closed and non-constant if and only if 
\begin{equation}\label{eq-necperiodiche}
0<h<mc^2,\quad \dfrac{\alpha^2}{c^2} < L^2 < \dfrac{\alpha^2m^2c^2}{m^2c^4-h^2}.
\end{equation}
(see Figure~\ref{fig-1}).
This means that the corresponding solution $x$ of \eqref{eq-keplero} is such that its modulus $r =  \vert x \vert$ is a non-constant periodic function. We point out that this does not imply that $x$ is a periodic solution: indeed, also the angular component 
$\vartheta$ plays a role (see Section~\ref{section-2.3}).
For completeness, we also observe that the case
\begin{equation*}
0<h<mc^2,\quad \dfrac{\alpha^2}{c^2} < L^2 = \dfrac{\alpha^2m^2c^2}{m^2c^4-h^2},
\end{equation*}
corresponds instead to the constant solution $(r^*,0)$ (that is, $x$ is a circular motion).

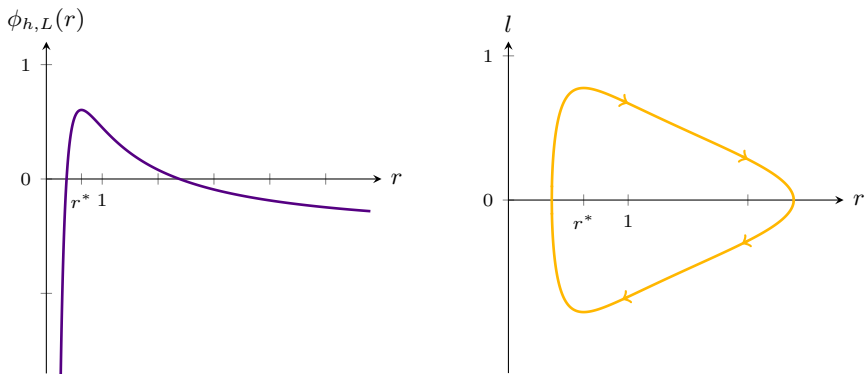
\begin{figure}[htb]
\definecolor{bdf-purple}{RGB}{85,0,130}
\definecolor{bdf-yellow}{RGB}{255,183,0}
\begin{tikzpicture}
\begin{axis}[
  tick label style={font=\scriptsize},
  axis y line=left, 
  axis x line=middle,
  xtick={0.628571,1,2,3,4,5},
  ytick={-1,0,1},
  xticklabels={$r^{*}$,$1$,,,,},
  yticklabels={,$0$,$1$},
  xlabel={\small $r$},
  ylabel={\small $\phi_{h,L}(r)$},
every axis x label/.style={
    at={(ticklabel* cs:1.0)},
    anchor=west,
},
every axis y label/.style={
    at={(ticklabel* cs:1.0)},
    anchor=south,
},
  width=6cm,
  height=6cm,
  xmin=0,
  xmax=6,
  ymin=-1.7,
  ymax=1.2]
\addplot[color=bdf-purple,line width=1pt,domain=0.1:5.8,smooth,samples=2000]{-0.51 - 0.44/x^2 + 1.4/x};
\end{axis}
\end{tikzpicture}
\quad\quad
\begin{tikzpicture}
\begin{axis}[
  tick label style={font=\scriptsize},
  axis y line=left, 
  axis x line=middle,
  xtick={0.628571,1,2},
  ytick={0,1},
  xticklabels={$r^{*}$,$1$,},
  yticklabels={$0$,$1$},
  xlabel={\small $r$},
  ylabel={\small $l$},
every axis x label/.style={
    at={(ticklabel* cs:1.0)},
    anchor=west,
},
every axis y label/.style={
    at={(ticklabel* cs:1.0)},
    anchor=south,
},
  width=6cm,
  height=6cm,
  xmin=0,
  xmax=2.8,
  ymin=-1.2,
  ymax=1.1]
\coordinate (A) at (axis cs:0.95,0.690036);
\coordinate (B) at (axis cs:1,0.67082);
\draw[->] [color=bdf-yellow,line width=1pt] (A)--(B);
\coordinate (C) at (axis cs:1.95,0.303703);
\coordinate (D) at (axis cs:2,0.282843);
\draw[->] [color=bdf-yellow,line width=1pt] (C)--(D);
\coordinate (E) at (axis cs:2,-0.282843);
\coordinate (F) at (axis cs:1.95,-0.303703);
\draw[->] [color=bdf-yellow,line width=1pt] (E)--(F);
\coordinate (G) at (axis cs:1,-0.67082);
\coordinate (H) at (axis cs:0.95,-0.690036);
\draw[->] [color=bdf-yellow,line width=1pt] (G)--(H);
\addplot +[mark=none,color=bdf-yellow,line width=1pt] coordinates {(0.362031, 0) (0.3635, 0.1)};
\addplot +[mark=none,color=bdf-yellow,line width=1pt] coordinates {(0.362031, 0) (0.3635, -0.1)};
\addplot +[mark=none,color=bdf-yellow,line width=1pt] coordinates {(2.38307, -0.01) (2.38307, 0.01)};
\addplot[color=bdf-yellow,line width=1pt,domain=0.36203:2.38307,smooth,samples=2000]{sqrt(-0.51 - 0.44/x^2 + 1.4/x)};
\addplot[color=bdf-yellow,line width=1pt,domain=0.36203:2.38307,smooth,samples=2000]{-sqrt(-0.51 - 0.44/x^2 + 1.4/x)};
\end{axis}
\end{tikzpicture}
\captionof{figure}{Graph of $\phi_{h,L}$ (on the left) and phase-portrait in the $(r,l)$-plane (on the right) with $\alpha=m=c=1$, and the energy $h$ and the angular momentum $L$ satisfying condition \eqref{eq-necperiodiche} (in the represented case $h=0.7$ and $L=1.2$).}
\label{fig-1}
\end{figure}

\medskip

In the following result we give an explicit formula for the minimal period of~$r$.

\begin{proposition}\label{teo-periodor} 
Let $x$ be a solution of \eqref{eq-keplero} with energy $h$ and angular momentum $L$.
If \eqref{eq-necperiodiche} holds, then the function $r=|x|$ is periodic, with minimal period given by
\begin{equation}\label{eq-periodor}
T_{h} = \dfrac{2\pi \alpha m^2 c^3}{(m^2c^4-h^2)^{\frac{3}{2}}}. 
\end{equation}
\end{proposition}

\begin{proof} Let us fix $h$ and $L$ such that condition \eqref{eq-necperiodiche} is satisfied, let $x$ be a solution of \eqref{eq-keplero} with energy $h$ and angular momentum $L$, and let $T_h$ be the period of $r=|x|$. Without loss of generality we can assume that $r(0)=r(T_h)=r_{m}$, where $r_m$ is the minimum value of $r$; by symmetry, we then have $r(T_h/2)=r_M$, being $r_M$ the maximum of $r$, and it also holds that $l(t)>0$, for every $t\in (0,T_h/2)$. Moreover, a simple computation proves that
\begin{equation}\label{eq-rminmax}
r_m=\dfrac{\alpha h-\sqrt{\Delta}}{m^2c^4-h^2},\qquad r_M=\dfrac{\alpha h+\sqrt{\Delta}}{m^2c^4-h^2},
\end{equation}
where $\Delta=\alpha^2 m^2c^4+L^2c^2(h^2-m^2c^4)$; notice that $\Delta>0$ by condition \eqref{eq-necperiodiche}.

Consider the equation of the trajectory given in \eqref{eq-traiett2}. Recalling \eqref{eq-funzphi} we have
\begin{equation}\label{eq-pr11}
\dfrac{l(t)}{\sqrt{\phi_{h,L}(r(t))}}=1, \quad \text{for every $t\in (0,T_h/2)$}.
\end{equation}
From the first equation in \eqref{eq-hs2} we have
\begin{equation}\label{eq-pr13}
l=m\dot{r}\, \sqrt{1+\dfrac{l^2+L^2/r^2}{m^2c^2}}=m\dot{r}\, \sqrt{1+\dfrac{\phi_{h,L}(r)+L^2/r^2}{m^2c^2}}.
\end{equation}
By replacing \eqref{eq-pr13} in \eqref{eq-pr11} we deduce that
\begin{equation*}
\dfrac{m\dot{r}(t)\, \sqrt{1+\dfrac{\phi_{h,L}(r(t))+L^2/(r(t))^{2}}{m^2c^2}}}{\sqrt{\phi_{h,L}(r(t))}}=1, \quad \text{for every $t\in (0,T_h/2)$;}
\end{equation*}
by integrating we infer that
\begin{align*}
T_h 
&=2m\int_0^{\frac{T_h}{2}} \dfrac{\dot{r}(t)\, \sqrt{1+\dfrac{\phi_{h,L}(r(t))+L^2/(r(t))^{2}}{m^2c^2}}}{\sqrt{\phi_{h,L}(r(t))}}\,\mathrm{d}t \\
&=2m\int_{r_m}^{r_M} \dfrac{\sqrt{1+\dfrac{\phi_{h,L}(u)+L^2/u^2}{m^2c^2}}}{\sqrt{\phi_{h,L}(u)}}\,\mathrm{d}u.
\end{align*}
Recalling \eqref{eq-funzphi}, by means of standard computations we obtain
\begin{equation}\label{eq-pr16}
\displaystyle T_h=\dfrac{2}{c}\, \int_{r_m}^{r_M} \dfrac{\alpha +hu}{\sqrt{(\alpha^2-L^2c^2)+2\alpha h u+(h^2-m^2c^4)u^2}}\,\mathrm{d}u;
\end{equation}
a primitive of the integrand is given by
\begin{equation*}
\dfrac{-2}{\sqrt{m^2c^4-h^2}}\, \biggl{[}\left(\alpha +\dfrac{1}{2}h(r_m+r_M)\right)\arctan \sqrt{\dfrac{r_M-u}{u-r_m}}+\dfrac{1}{2}h\sqrt{(r_M-u)(u-r_m)}\biggr{]}
\end{equation*}
and hence from \eqref{eq-pr16} we obtain
\begin{equation*}
\displaystyle T_h=\dfrac{2\pi}{c\sqrt{m^2c^4-h^2}}\, \left(\alpha +\dfrac{1}{2}h(r_m+r_M)\right).
\end{equation*}
Taking into account \eqref{eq-rminmax}, we can conclude.
\end{proof}

\subsection{Periodic and quasi-periodic solutions}\label{section-2.3}

Proposition~\ref{teo-periodor} guarantees the periodicity of the modulus $r$ of the solutions $x$ of \eqref{eq-keplero} under condition \eqref{eq-necperiodiche}. To ensure that $x$ is a periodic function, we further need some information on the angular part $\vartheta$.

We first show that it is possible to write the trajectory of $x$ in the polar form $r=\rho(\vartheta)$, for some function $\rho$.
To this end, we first observe that if $x = re^{i\vartheta}$ is a solution of \eqref{eq-keplero} then
\begin{equation*}
\dot{x}=\dot{r}\, e^{i\vartheta}+r\dot{\vartheta}\, ie^{i\vartheta}
\end{equation*}
and so, by \eqref{eq-momentop}, $p=p_r e^{i\vartheta} + p_\vartheta i e^{i\vartheta}$, where
\begin{equation}\label{eq-plungosol}
p_r = \dfrac{m\dot{r}}{\sqrt{1-|\dot{x}|^2/c^2}}, \qquad p_\vartheta = \dfrac{mr\dot{\vartheta}}{\sqrt{1-|\dot{x}|^2/c^2}}.
\end{equation}
Comparing \eqref{eq-momentop} with the second relation in \eqref{eq-cambiovar2}, we deduce that
\begin{equation}\label{eq-inver}
\dfrac{mr\dot{\vartheta}}{\sqrt{1-|\dot{x}|^2/c^2}}=\dfrac{\Phi}{r},
\end{equation}
so that, since $\Phi$ is a constant of motion, we infer that $\dot{\vartheta}$ has a constant sign along the motion. Therefore, $\vartheta$ is invertible with respect to the time. Let $t=t(\vartheta)$ be its inverse.

From \eqref{eq-plungosol} we deduce that
\begin{equation*}
\dfrac{p_r}{p_\vartheta}=\dfrac{\dot{r}}{r\dot{\vartheta}} = \dfrac{1}{\rho}\, \dfrac{\mathrm{d}\rho}{\mathrm{d}\vartheta},
\end{equation*}
where we have set $\rho (\vartheta)=r(t(\vartheta))$.
Using \eqref{eq-cambiovar2} together with \eqref{eq-inver}, this implies that
\begin{equation*}
p_r=\dfrac{\Phi}{\rho^2}\, \dfrac{\mathrm{d}\rho}{\mathrm{d}\vartheta}
\end{equation*}
and finally
\begin{equation*}
|p|^2=\left(\dfrac{\Phi}{\rho^2}\, \dfrac{\mathrm{d}\rho}{\mathrm{d}\vartheta}\right)^{\!2}+\dfrac{\Phi^2}{\rho^2}.
\end{equation*}
As a consequence, recalling \eqref{eq-hamiltoniana} and denoting again by $h$ the energy of $x$ and by $L$ its angular momentum, we infer that
\begin{equation*}
mc^2 \sqrt{1+\dfrac{\left(\dfrac{L}{\rho^2}\, \dfrac{\mathrm{d}\rho}{\mathrm{d}\vartheta}\right)^{\!2}+\dfrac{L^2}{\rho^2}}{m^2c^2}}=\dfrac{\alpha}{\rho}+h.
\end{equation*}
We then deduce that
\begin{equation*}
c^2L^2\, \left(\dfrac{1}{\rho^2}\, \dfrac{\mathrm{d}\rho}{\mathrm{d}\vartheta}\right)^2+\dfrac{c^2L^2-\alpha^2}{\rho^2}-\dfrac{2\alpha h}{\rho}+m^2c^4-h^2=0,
\end{equation*}
so that, passing to the usual $\rho$-inverse variable $s=1/\rho$ (the so-called Clairaut's change of variable), we obtain
\begin{equation}\label{eq-pr22}
c^2L^2\, \left(\dfrac{\mathrm{d}s}{\mathrm{d}\vartheta}\right)^2+(c^2 L^2-\alpha^2)s^2-2\alpha hs+m^2c^4-h^2=0.
\end{equation}
Differentiating \eqref{eq-pr22}, we then deduce
\begin{equation*}
c^2L^2\, \dfrac{\mathrm{d}^2 s}{\mathrm{d}\vartheta^2}+(c^2L^2-\alpha^2)s-\alpha h=0,
\end{equation*}
which can be solved to obtain
\begin{equation*}
s(\vartheta)=A\cos \biggl{(} \sqrt{1-\dfrac{\alpha^2}{c^2L^2}} (\vartheta-\vartheta_0) \biggr{)}+\dfrac{\alpha h}{c^2L^2-\alpha^2},
\end{equation*}
where $A$ and $\vartheta_0$ are arbitrary real constants. 
However, in order for $s(\vartheta)$ to solve the original equation \eqref{eq-pr22} we
need to require
\begin{equation*}
A=\dfrac{\sqrt{\alpha^2 m^2c^4+(h^2-m^2c^4)c^2L^2}}{c^2L^2-\alpha^2};
\end{equation*}
notice that $m^2c^4+(h^2-m^2c^4)c^2L^2>0$ by \eqref{eq-necperiodiche}. Moreover, 
\begin{equation*}
A < \dfrac{\alpha h}{c^2L^2-\alpha^2},
\end{equation*}
again by condition \eqref{eq-necperiodiche}. Then, $s(\vartheta) > 0$ for every $\vartheta \in \mathbb{R}$ and we conclude that
\begin{equation}\label{eq-rdipendedatheta}
\rho(\vartheta)=\dfrac{1}{\dfrac{\sqrt{\alpha^2 m^2c^4+(h^2-m^2c^4)c^2L^2}}{c^2L^2-\alpha^2}\, \cos \left(\sqrt{1-\dfrac{\alpha^2}{c^2L^2}} (\vartheta-\vartheta_0)\right)+\dfrac{\alpha h}{c^2L^2-\alpha^2}}.
\end{equation}
This provides the explicit polar equation for the trajectory of a solution $x$ of \eqref{eq-keplero}, when \eqref{eq-necperiodiche} is satisfied.

Wishing to investigate the periodicity of $x$, the crucial observation is that the function $\rho$ is periodic of minimal period
\begin{equation}\label{eq-deltatheta0}
\Delta\vartheta= \dfrac{2\pi}{\sqrt{1-\dfrac{\alpha^2}{c^2L^2}}}.
\end{equation}
As a consequence, the planar curve of polar equation $r=\rho(\vartheta)$ is closed if and only if $\Delta\vartheta$ is commensurable with $2\pi$, i.e.~if and only if there exist $k, n\in \mathbb{N}$ such that
\begin{equation}\label{eq-commensurabile}
\sqrt{1-\dfrac{\alpha^2}{c^2L^2}}=\dfrac{n}{k},
\end{equation}
where $n$ and $k$ are relatively prime. 
Precisely, if \eqref{eq-commensurabile} holds true, then 
the planar curve $\rho(\vartheta) e^{i\vartheta}$ 
is periodic with minimal period $2\pi k$, which is exactly $n$ times the period of the radial component $\rho(\vartheta)$. 
Then, the solution
\begin{equation*}
x(t)=\rho(\vartheta(t))\, e^{i\vartheta(t)}=r(t\, )e^{i\vartheta (t)}
\end{equation*}
has minimal period equal to $n T_h$, where $T_h$ is defined in \eqref{eq-periodor}; moreover, on the periodicity interval $[0,n T_h]$ the angular coordinate covers the angle
\begin{equation*}
n \Delta \vartheta = 2\pi k,
\end{equation*}
so that the winding number of $x$ equals to $\pm k$, depending on the sign of the angular momentum (the name $\Delta\vartheta$ is indeed justified by the fact that this quantity is exactly the angular variation made by $\vartheta = \vartheta(t)$ in the time from $t = 0$ to $t = T_h$, that is $\Delta\vartheta = \vartheta(T_h) - \vartheta(0)$).

Recalling the definition of $r^*$ given in \eqref{eq-rstarbasso}, we can thus state the following result.

\begin{proposition}\label{teo-periodox}
Let $x$ be a solution of \eqref{eq-keplero} with energy $h$ and angular momentum $L$.
If \eqref{eq-necperiodiche} holds true and \eqref{eq-commensurabile} is satisfied for some coprime integers $n$ and $k$,
then $x$ is periodic of minimal period 
\begin{equation*}
n T_{h} = \dfrac{2\pi \alpha m^2 c^3 n}{(m^2c^4-h^2)^{\frac{3}{2}}}.
\end{equation*}
Moreover, the following holds:
\begin{itemize}
\item[$(i)$] $|x|$ assumes the value $r^{*}$ exactly $2n$ times in the interval $[0,n T_h)$, always with non-zero derivative;
\item[$(ii)$] $x$ has winding number equal to $\mathrm{sgn}(L) \cdot k$ in the interval $[0, n T_h)$.
\end{itemize}
\end{proposition}

In Figure~\ref{fig-2} a qualitative picture of some trajectories in the $x$-plane is drawn.

\begin{figure}[htb]
\centering
\begin{subfigure}[t]{.45\textwidth}
\centering
\begin{tikzpicture}[scale=0.51]
\definecolor{bdf-blue}{RGB}{0,51,153}
\definecolor{bdf-red}{RGB}{153,51,0}
\begin{polaraxis}[
xtick={0,45,90,135,180,225,270,315,360},
xticklabels={$\,0$,$\dfrac{\pi}{4}$,$\dfrac{\pi}{2}$,$\dfrac{3}{4}\pi$,$\pi\,$,$\dfrac{5}{4}\pi$,$\dfrac{3}{2}\pi$,$\dfrac{7}{4}\pi$,},
ytick={},ylabel={},yticklabels={},
major tick length=0ex,
axis line style = {draw=gray,line width=0.05pt},
]
\addplot[color=bdf-blue,line width=1.3pt,domain=0:720,smooth,samples=2000]{1/(1.69706*cos(1*x/2)+2.1)};
\addplot[color=bdf-red,line width=1pt,dash pattern=on 4pt off 4pt,domain=0:360,samples=600]{0.47619};
\end{polaraxis}
\end{tikzpicture}
\caption{Trajectory for $(n,k)=(1,2)$.}
\end{subfigure}
\quad
\begin{subfigure}[t]{.45\textwidth}
\centering
\begin{tikzpicture}[scale=0.51]
\definecolor{bdf-blue}{RGB}{0,51,153}
\definecolor{bdf-red}{RGB}{153,51,0}
\begin{polaraxis}[
xtick={0,45,90,135,180,225,270,315,360},
xticklabels={$\,0$,$\dfrac{\pi}{4}$,$\dfrac{\pi}{2}$,$\dfrac{3}{4}\pi$,$\pi\,$,$\dfrac{5}{4}\pi$,$\dfrac{3}{2}\pi$,$\dfrac{7}{4}\pi$,},
ytick={},ylabel={},yticklabels={},
major tick length=0ex,
axis line style = {draw=gray,line width=0.05pt},
]
\addplot[color=bdf-blue,line width=1.3pt,domain=0:1080,smooth,samples=2000]{1/(5.22303*cos(1*x/3)+5.6)};
\addplot[color=bdf-red,line width=1pt,dash pattern=on 4pt off 4pt,domain=0:360,samples=600]{0.178571};
\end{polaraxis}
\end{tikzpicture}
\caption{Trajectory for $(n,k)=(1,3)$.}
\end{subfigure}
\vspace{10pt}
\\
\begin{subfigure}[t]{.45\textwidth}
\centering
\begin{tikzpicture}[scale=0.51]
\definecolor{bdf-blue}{RGB}{0,51,153}
\definecolor{bdf-red}{RGB}{153,51,0}
\begin{polaraxis}[
xtick={0,45,90,135,180,225,270,315,360},
xticklabels={$\,0$,$\dfrac{\pi}{4}$,$\dfrac{\pi}{2}$,$\dfrac{3}{4}\pi$,$\pi\,$,$\dfrac{5}{4}\pi$,$\dfrac{3}{2}\pi$,$\dfrac{7}{4}\pi$,},
ytick={},ylabel={},yticklabels={},
major tick length=0ex,
axis line style = {draw=gray,line width=0.05pt},
]
\addplot[color=bdf-blue,line width=1.3pt,domain=0:1080,smooth,samples=2000]{1/(0.357946*cos(2*x/3)+0.875)};
\addplot[color=bdf-red,line width=1pt,dash pattern=on 4pt off 4pt,domain=0:360,samples=600]{1.14286};
\end{polaraxis}
\end{tikzpicture}
\caption{Trajectory for $(n,k)=(2,3)$.}
\end{subfigure} 
\quad
\begin{subfigure}[t]{.45\textwidth}
\centering
\begin{tikzpicture}[scale=0.51]
\definecolor{bdf-blue}{RGB}{0,51,153}
\definecolor{bdf-red}{RGB}{153,51,0}
\begin{polaraxis}[
xtick={0,45,90,135,180,225,270,315,360},
xticklabels={$\,0$,$\dfrac{\pi}{4}$,$\dfrac{\pi}{2}$,$\dfrac{3}{4}\pi$,$\pi\,$,$\dfrac{5}{4}\pi$,$\dfrac{3}{2}\pi$,$\dfrac{7}{4}\pi$,},
ytick={},ylabel={},yticklabels={},
major tick length=0ex,
axis line style = {draw=gray,line width=0.05pt},
]
\addplot[color=bdf-blue,line width=1.3pt,domain=0:1800,smooth,samples=2000]{1/(3.29061*cos(2*x/5)+3.675)};
\addplot[color=bdf-red,line width=1pt,dash pattern=on 4pt off 4pt,domain=0:360,samples=600]{0.272109};
\end{polaraxis}
\end{tikzpicture}
\caption{Trajectory for $(n,k)=(2,5)$.}
\end{subfigure} 
\vspace{10pt}
\\
\begin{subfigure}[t]{.45\textwidth}
\centering
\begin{tikzpicture}[scale=0.51]
\definecolor{bdf-blue}{RGB}{0,51,153}
\definecolor{bdf-red}{RGB}{153,51,0}
\begin{polaraxis}[
xtick={0,45,90,135,180,225,270,315,360},
xticklabels={$\,0$,$\dfrac{\pi}{4}$,$\dfrac{\pi}{2}$,$\dfrac{3}{4}\pi$,$\pi\,$,$\dfrac{5}{4}\pi$,$\dfrac{3}{2}\pi$,$\dfrac{7}{4}\pi$,},
ytick={},ylabel={},yticklabels={},
major tick length=0ex,
axis line style = {draw=gray,line width=0.05pt},
]
\addplot[color=bdf-blue,line width=1.3pt,domain=0:1800,smooth,samples=2000]{1/(0.801234*cos(3*x/5)+1.24444)};
\addplot[color=bdf-red,line width=1pt,dash pattern=on 4pt off 4pt,domain=0:360,samples=600]{0.803571};
\end{polaraxis}
\end{tikzpicture}
\caption{Trajectory for $(n,k)=(3,5)$.}
\end{subfigure} 
\quad
\begin{subfigure}[t]{.45\textwidth}
\centering
\begin{tikzpicture}[scale=0.51]
\definecolor{bdf-blue}{RGB}{0,51,153}
\definecolor{bdf-red}{RGB}{153,51,0}
\begin{polaraxis}[
xtick={0,45,90,135,180,225,270,315,360},
xticklabels={$\,0$,$\dfrac{\pi}{4}$,$\dfrac{\pi}{2}$,$\dfrac{3}{4}\pi$,$\pi\,$,$\dfrac{5}{4}\pi$,$\dfrac{3}{2}\pi$,$\dfrac{7}{4}\pi$,},
ytick={},ylabel={},yticklabels={},
major tick length=0ex,
axis line style = {draw=gray,line width=0.05pt},
]
\addplot[color=bdf-blue,line width=1.3pt,domain=0:2880,smooth,samples=2000]{1/(0.629971*cos(5*x/8)+1.092)};
\addplot[color=bdf-red,line width=1pt,dash pattern=on 4pt off 4pt,domain=0:360,samples=600]{0.915751};
\end{polaraxis}
\end{tikzpicture}
\caption{Trajectory for $(n,k)=(5,8)$.}
\end{subfigure} 
\caption{Dynamics of the solutions of \eqref{eq-keplero} in the $x$-plane: the blue line represents the trajectory of $x$ in the polar form $r=\rho(\vartheta)$ given in \eqref{eq-rdipendedatheta} with $\alpha=m=c=1$ and $h=0.7$, the red dashed line represents the value $r^{*}$ given in \eqref{eq-rstarbasso}. For some values of $(n,k)$ the figure shows that conditions $(i)$ and $(ii)$ of Proposition~\ref{teo-periodox} are satisfied.}
\label{fig-2}
\end{figure}
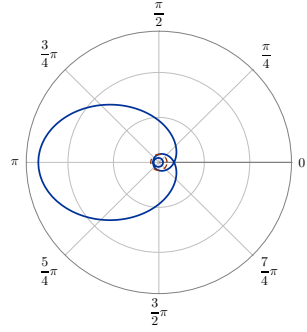
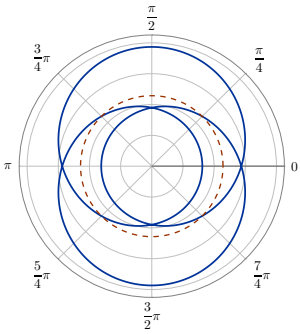
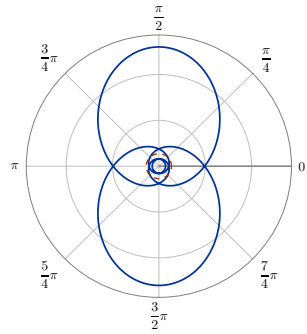
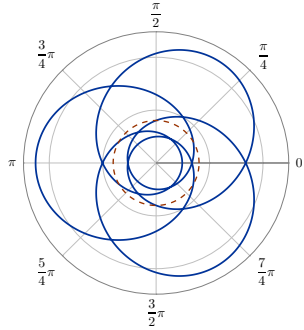
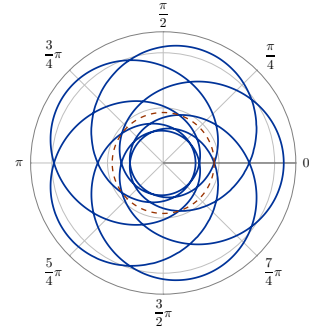

\begin{remark}[Quasi-periodic solutions]\label{rem-quasiperiodiche}
The case when 
\begin{equation}\label{eq-noncommens}
\sqrt{1-\dfrac{\alpha^2}{c^2L^2}}\notin \mathbb{Q}
\end{equation}
gives rise to quasi-periodic solutions: indeed, assume that \eqref{eq-noncommens} holds true and let
\begin{equation*}
\omega = \dfrac{2\pi}{T_h\sqrt{1-\dfrac{\alpha^2}{c^2L^2}}}.
\end{equation*}
Then, for every solution $x$ with energy $h$ and angular momentum $L$, from the facts that $r=|x|$ has minimal period $T_h$ and that $r=\rho(\vartheta)$, we deduce that the function 
$t \mapsto \vartheta(t) - \omega t$ 
is $T_h$-periodic. As a consequence, we can write
\begin{equation*}
x(t)=r(t)\, e^{i\vartheta(t)}=r(t)\, e^{i(\vartheta(t)-\omega t)}\, e^{i\omega t}={\hat x}(t)\, e^{i\omega t},\quad \text{for every $t\in \mathbb{R}$,}
\end{equation*}
since ${\hat x}$ is $T_h$-periodic and 
\begin{equation*}
\dfrac{2\pi}{\omega}\, \dfrac{1}{T_h}=\sqrt{1-\dfrac{\alpha^2}{c^2L^2}}\notin \mathbb{Q},
\end{equation*}
by \eqref{eq-noncommens}, the conclusion follows.
\end{remark}

\subsection{Action-angle variables}\label{section-2.4}

In this section, we construct action-angle variables for the Hamiltonian system \eqref{eq-hs}. Let us consider the open set
\begin{equation*}
\Lambda_0 = \left\{ (x,p) \in (\mathbb{R}^2 \setminus \{0\}) \times \mathbb{R}^2 \colon 
\begin{array}{l}
0 < H_0(x,p) < mc^2 \vspace{3pt}\\
L_0(x,p) >0 \vspace{3pt}\\ 
\dfrac{\alpha^2}{c^2} < L_0(x,p)^2 < \dfrac{\alpha^2 m^2 c^2}{m^2 c^4 - (H_0(x,p))^2}
\end{array}
\right\}
\end{equation*}
and set $\Omega_0 = \Psi^{-1}(\Lambda_0)$, where $\Psi$ is defined in \eqref{eq-cambiovar1} and \eqref{eq-cambiovar2}. More explicitly,
\begin{equation*}
\Omega_0 = \left\{ (r,\vartheta,l,\Phi) \in \Omega \colon
\begin{array}{l}
0 < \mathcal{H}_0(r,\vartheta,l,\Phi) < mc^2 \vspace{3pt}\\
\mathcal{L}_0(r,\vartheta,l,\Phi) >0 \vspace{3pt}\\ 
\dfrac{\alpha^2}{c^2} < \mathcal{L}_0(r,\vartheta,l,\Phi)^2 < \dfrac{\alpha^2 m^2 c^2}{m^2 c^4 - (\mathcal{H}_0(r,\vartheta,l,\Phi))^2}
\end{array}
\right\}.
\end{equation*}
As discussed in the previous section, every $(x^*,p^*) \in \Lambda_0$ (or, equivalently, every $(r^*,\vartheta^*,l^*,\Phi^*) \in \Omega_0$) is the initial condition of a quasi-periodic solution of \eqref{eq-hs} (equivalently, of \eqref{eq-hs2}); such a solution is actually a periodic one whenever 
\eqref{eq-commensurabile} is satisfied. In any case, the corresponding projection in the $(r,l)$-plane is a (non-constant) closed orbit.

For every values $h$ and $L$, with $L > 0$, satisfying \eqref{eq-necperiodiche}, the (compact and connected) level set
\begin{align*}
\Omega_0(h,L) & = \bigl{\{} (r,\vartheta,l,\Phi) \in \Omega_0 \colon \mathcal{H}_0(r,\vartheta,l,\Phi) = h, \, \mathcal{L}_0(r,\vartheta,l,\Phi) = L \bigr{\}} \\
& = \bigl{\{} (r,\vartheta,l,L) \in \Omega_0 \colon \mathcal{H}_0(r,\vartheta,l,L) = h \bigr{\}}
\end{align*}
is diffeomorphic to a torus $\mathbb{T}^2$: this follows from the Liouville--Arnold theorem (see \cite[Chapter~10]{Ar-89} or \cite[Chapter~3.1]{MoZe-05}), since on this set
the integrals $\mathcal{H}_0$ and $\mathcal{L}_0$ are linearly independent and in involution, as it is easy to check.
Following the proof of this theorem, action-angle coordinates can be constructed in a standard way.

The precise definition of the angles $(\varphi_1,\varphi_2)$ is not needed for our purposes, and we refer again to \cite[Chapter~10]{Ar-89} for the details. On the other hand, the two action variables $I_1, I_2$ are defined through the formulas
\begin{equation*}
I_i(h,L) = \frac{1}{2\pi} \oint_{\gamma_i} \left( l \,\mathrm{d}r + L \,\mathrm{d}\vartheta\right), \qquad i=1,2, 
\end{equation*}
where $\gamma_1,\gamma_2$ are two independent cycles on the torus $\Omega_0(h,L)$.

The cycle $\gamma_1$ can be obtained in this way: fixed a point $(r,\vartheta,l,L) \in \Omega_0(h,L)$ (for simplicity, we take $\vartheta=0$), we first follow the evolution through the Hamiltonian flow for the period $T_h$ defined in \eqref{eq-periodor} (in such a way that the projection on the $(r,l)$-plane completes a full turn) and later we move $\vartheta$ with velocity equal to $-1$ until we reach $\vartheta = 2\pi$.
That is 
\begin{equation*}
\gamma_1(t) = \begin{cases}
(r(t),\vartheta(t),l(t),\Phi(t)), & \text{if $0 \leq t \leq T_h$,} \\
(r(T_h), - (t-T_h) + \Delta\vartheta,l(T_h),\Phi(T_h)), & \text{if $T_h \leq t \leq T_h + \beta$,}
\end{cases}
\end{equation*}
where $(r(t),\vartheta(t),l(t),\Phi(t))$ is the solution of \eqref{eq-hs2} with $(r(0),\vartheta(0),l(0),\Phi(0)) = (r,0,l,L) \in \Omega_0(h,L)$ (notice that $\Phi(t) \equiv L$), $\Delta\vartheta$ is as in \eqref{eq-deltatheta0} (recall that $\Delta\vartheta = \vartheta(T_h)$), and 
$\beta = \Delta\vartheta-2\pi$.
A simple computation yields
\begin{equation}\label{eq-defi1}
I_1(h,L) = \frac{\mathcal{A}(h,L)}{2\pi} + L,
\end{equation} 
where 
\begin{equation*}
\mathcal{A}(h,L) = \frac{2\pi}{c} \left( \frac{\alpha h}{\sqrt{m^2 c^4 - h^2}} - \sqrt{c^2 L^2 - \alpha^2}\right)
\end{equation*}
is the area of the region in the $(r,l)$-plane enclosed by the curve $\mathcal{H}_0(r,\vartheta,l,L) = h$.

The cycle $\gamma_2$ is simply defined as $\gamma_2(t) = (r,t,l,L)$ for $t \in [0,2\pi]$, where again $(r,0,l,L) \in \Omega_0(h,L)$, so that
\begin{equation}\label{eq-defi2}
I_2(h,L) = L.
\end{equation}

As is well known, the change of variables $(r,\vartheta,l,L) \mapsto (I_1,I_2,\varphi_1,\varphi_2)$ is symplectic, that is
\begin{equation*}
\mathrm{d}r \wedge \mathrm{d}l + \mathrm{d}\vartheta \wedge \mathrm{d}\Phi = \mathrm{d}I_1 \wedge \mathrm{d}\varphi_1 + \mathrm{d}I_2 \wedge \mathrm{d}\varphi_2,
\end{equation*}
and in principle can be defined on a neighborhood of $\Omega_0(h,L)$ in $\Omega_0$. 
As a consequence of this construction, however, it turns out that this change of variables provides a global symplectic diffeomoprhism
\begin{equation*}
\Sigma\colon \Omega_0 \to \mathcal{D} \times \mathbb{T}^2, \qquad (r,\vartheta,l,L) \mapsto (I_1,I_2,\varphi_1,\varphi_2),
\end{equation*}
where $\mathcal{D}$ is the open set of $\mathbb{R}^2$ in which $I_1,I_2$ are allowed to vary (we omit the explicit definition, which is complicated and useless for our purposes).
In this domain, the Hamiltonian in action-angle coordinates can be obtained by inverting \eqref{eq-defi1} (and recalling \eqref{eq-defi2}); after some computations, we obtain
\begin{equation}\label{eq-hamactang}
\mathcal{K}_0(I_1,I_2,\varphi_1,\varphi_2) = \mathcal{K}_0(I_1,I_2) = mc^2 \dfrac{I_1 - I_2 + \dfrac{1}{c}\sqrt{c^2I_2^2 - \alpha^2}}{\sqrt{(I_1-I_2)^2 + I_2^2+\dfrac{2}{c}(I_1-I_2) \sqrt{c^2 I_2^2 - \alpha^2}}}.
\end{equation}

\begin{remark}[Passing to the non-relativistic limit]\label{rem-rlimite}
It can be interesting to observe that all the theory of the Kepler problem (see, for instance, \cite[Chapter~1]{Po-76}) can be recovered by passing to the non-relativistic limit $c \to +\infty$. Indeed, writing 
\begin{equation*}
H_0 = \mathcal{E} + mc^2,
\end{equation*}
where $H_0$ is the relativistic energy (that is, the Hamiltonian \eqref{eq-hamiltoniana}) and $\mathcal{E}$ is the energy difference from the particle rest energy,
and denoting by 
\begin{equation*}
\mathcal{E}_{nr} = \frac{1}{2} m | \dot x |^2 - \frac{\alpha}{| x |}
\end{equation*}
the usual non-relativistic energy, a simple computation shows that $\mathcal{E} \to \mathcal{E}_{nr}$ as $c \to +\infty$.
Analogously, denoting by
\begin{equation*}
L_{nr} = \langle Jx , m \dot x \rangle
\end{equation*}
the usual non-relativistic angular momentum, it holds that $L \to L_{nr}$ as $c \to +\infty$.

Using the above relationships, by passing to the limit in \eqref{eq-necperiodiche}, it follows that the admissible values
of $\mathcal{E}_{nr}$ and $L_{nr}$ for the existence of (non-circular) closed orbits in the $(r,\dot r)$-plane satisfy the well-known conditions
\begin{equation*}
\mathcal{E}_{nr} < 0 \quad \text{ and } \quad 0 < L_{nr}^2 < \frac{\alpha^2 m}{-2 \mathcal{E}_{nr}}.
\end{equation*}
By taking the limit in \eqref{eq-deltatheta0} it follows that $\Delta\vartheta = 2\pi$, 
so that all these solutions are periodic also in the $x$-plane, with winding number $\pm 1$ in the period
\begin{equation*}
T_{h,nr} = \frac{2\pi \alpha m^2}{(-2\mathcal{E}_{nr})^{\frac{3}{2}}}.
\end{equation*}
This can be obtained by passing to the limit in \eqref{eq-periodor}, and is nothing but the third Kepler's law; moreover, taking the limit in the polar equation \eqref{eq-rdipendedatheta} provides
\begin{equation*}
\rho(\vartheta) = \frac{L_{nr}^2}{\sqrt{\alpha^2 m^2 + 2m \mathcal{E}_{nr} L_{nr}^2} \cos(\vartheta-\vartheta_0) + \alpha m},
\end{equation*}
which is the polar equation of a Keplerian ellipse.

Finally, writing in \eqref{eq-hamactang} $\mathcal{K}_0 = \tilde{\mathcal{K}}_0 + mc^2$, a simple computation shows that the non-relativistic limit of $\tilde{\mathcal{K}}_0$ becomes
\begin{equation*}
\mathcal{K}_{0,nr}(I_{1,nr}) = -\frac12 \frac{\alpha^2 m}{I_{1,nr}^2}, 
\end{equation*}
which corresponds to the usual expression of the Kepler Hamiltonian in action-angle coordinates $I_{i,nr}$ and $\varphi_{i,nr}$ (the so-called Delaunay variables, see \cite[Chapter~3.2]{MoZe-05}). 
\end{remark}

\section{The perturbed problem} \label{section-3}

In this section, we deal with the perturbed problem \eqref{eq-completa}. 
We first recall a recent existence result for $T$-periodic solutions of a perturbed Hamiltonian system (Section~\ref{section-3.1}) which is subsequently exploited to prove our main result (Section~\ref{section-3.2}).

\subsection{Periodic perturbations of completely integrable Hamiltonian systems}\label{section-3.1}

In this section, we briefly describe a result from \cite{FoGaGi-16} (see also \cite{AmCoEk-87,BeKa-87,Ch-92}), dealing with the existence of $T$-periodic solutions for the Hamiltonian system
\begin{equation}\label{hs}
\begin{cases}
\, \dot I = - \varepsilon \, \nabla_\varphi \mathcal{R}(t,\varphi\,I), \\
\, \dot \varphi = \nabla \mathcal{K}(I) + \varepsilon \, \nabla_I \mathcal{R}(t,\varphi,I),
\end{cases}
\end{equation}
where $\mathcal{K}\colon \mathcal{D} \to \mathbb{R}$ is twice continuously differentiable, $\mathcal{R}\colon \mathbb{R} \times \mathcal{D} \times \mathbb{T}^N \to \mathbb{R}$ is continuous in all its variables, continuously differentiable in $(I,\varphi)$ and $T$-periodic in $t$,
and $\varepsilon > 0$ is a (small) parameter. Throughout this section, it is convenient to mean $\mathbb{T}^N$ as the set 
$[0,2\pi)^N$ with the standard torus topology; the natural covering projection $\mathbb{R}^N \ni \tilde\varphi \mapsto \varphi \in \mathbb{T}^N$ will be denoted by $\Pi$.

System \eqref{hs} arises as a time-periodic perturbation of the completely integrable Hamiltonian system
in action-angle coordinates
\begin{equation}\label{hs0}
\begin{cases}
\, \dot I = 0, \\
\, \dot \varphi = \nabla \mathcal{K}(I). 
\end{cases}
\end{equation}
The dynamics of the above system is well-known, consisting of periodic and quasi-periodic solutions lying on invariant tori. 

In the following, we assume that, for some $\bar I \in \mathcal{D}$,
\begin{equation}\label{toro}
T \nabla \mathcal{K}(\bar I) \in 2\pi \mathbb{Z}^N 
\end{equation}
and that $T$ is the minimum positive real number satisfying the above property.
This means that the invariant torus $\mathbb{T}^N \times \{\bar I\}$ is filled by periodic orbits with minimal period $T$.

The next result, which is a corollary of \cite[Corollary~2.2]{FoGaGi-16}, ensures the survival of $N+1$ of the above $T$-periodic solutions when $\varepsilon$ is small enough, provided that a non-degeneracy condition is satisfied. 

\begin{theorem}\label{fgg}
Let us suppose that \eqref{toro} holds true and that
\begin{equation}\label{hessiano}
\det \nabla^2 \mathcal{K}(\bar I) \neq 0.
\end{equation}
Then, for every $\sigma > 0$ there exists $\varepsilon^* > 0$ such that, for every $\varepsilon \in (0,\varepsilon^*)$, system
\eqref{hs} has at least $N+1$ solutions with period $T$, satisfying
\begin{equation}\label{vicino}
\max_{t \in [0,T]} | \tilde\varphi(t) - \tilde\varphi(0) - t \nabla \mathcal{K}(\bar I) | + \max_{t \in [0,T]} |I(t) - \bar I | \leq \sigma,
\end{equation}
where $\tilde\varphi \colon [0,T] \to \mathbb{R}^N$ is a lifting of $\varphi$ with respect to the covering projection $\Pi$.
\end{theorem}

Condition \eqref{vicino} means that the function $(I(t),\tilde\varphi(t))$, which is a solution of the lifting of system \eqref{hs} to the covering space $\mathcal{D} \times \mathbb{R}^N$, remains arbitrarily close, in the $\mathcal{C}^0$-norm,
to the function $(\bar I, \tilde \varphi(0) + t \nabla \mathcal{K}(\bar I))$, solving the lifting of the unperturbed system \eqref{hs0}. By considering the distance in $\mathbb{T}^N$ given by
\begin{equation*}
\mathrm{d}_{\mathbb{T}^N}(\varphi,\psi) = \min_{k \in \{-1,0,1\}^N} | \varphi - \psi + 2k\pi |,
\end{equation*}
it is easy to see that, whenever $\sigma < 2\pi$, \eqref{vicino} implies that
\begin{equation}\label{vicinotoro}
\max_{t \in [0,T]} \mathrm{d}_{\mathbb{T}^N}(\varphi(t),\Pi(\tilde\varphi(0) + t \nabla \mathcal{K}(\bar I))) + \max_{t \in [0,T]} |I(t) - \bar I | \leq \sigma,
\end{equation}
that is, $(I(t),\varphi(t))$ remains close, in the $\mathcal{C}^0$-norm on $\mathbb{R}^2 \times \mathbb{T}^2$, to a solution 
of the unperturbed system \eqref{hs0}. This remark will be crucial in the next section.

\subsection{The main result: statement and proof}\label{section-3.2}

In this section, we state and prove our main result dealing with the existence of $T$-periodic solutions for the perturbed relativistic Kepler problem
\begin{equation}\label{eq-pert}
\dfrac{\mathrm{d}}{\mathrm{d}t}\left(\dfrac{m\dot{x}}{\sqrt{1-|\dot{x}|^2/c^2}}\right)=-\alpha\, \dfrac{x}{|x|^3}+\varepsilon \, \nabla_x U(t,x), \qquad
x \in \mathbb{R}^2 \setminus \{0\},
\end{equation}
where $m,c, \alpha$ are fixed positive constants and $\varepsilon > 0$ is a (small) parameter.
As for the perturbation function $U = U(t,x)$, we assume that $U\colon \mathbb{R} \times (\mathbb{R}^2 \setminus \{0\}) \to \mathbb{R}$ is continuous
in $(t,x)$, continuously differentiable in $x$ and $T$-periodic in $t$: henceforth, we denote this class of functions by $\mathcal{C}_T^{0,1}(\mathbb{R} \times (\mathbb{R}^2 \setminus \{0\}))$.

To state our result, we further need to introduce some constants.
As a first step, we define, for every integer $n \geq 1$, the value
\begin{equation*}
T^*_n = \dfrac{2\pi n \alpha}{mc^3}.
\end{equation*}
Notice that $T^*_1$ is nothing but the number appearing in condition \eqref{bound-T}: it corresponds to the limit value 
of the function $T_h$ given in \eqref{eq-periodor} when $h \to 0^+$.

Then, we fix $T > T^*_1$ and we define, for every integer $n \geq 1$ such that
$T > T^*_n$, the integer $k^*_{T,n} \geq n$ as follows.
Let $h_{T,n} \in (0,mc^2)$ be the unique value such that 
\begin{equation*}
n T_{h_{T,n}} = T,
\end{equation*}
notice that the existence and the uniqueness of $h_{T,n}$ follows from the facts that $T > T^*_n$ and that $T_h$ is a strictly increasing function (with respect to $h$) 
with $T_h \to +\infty$ for $h \to (mc^2)^-$.
Then, we define $k^*_{T,n}$ as the smallest integer number such that
\begin{equation*}
k^*_{T,n} > \frac{mc^2 n}{h_{T,n}}.
\end{equation*}
Notice that $k^*_{T,n} > n$ so that, in particular, $k^*_{T,n} \geq 2$. 

Finally, for $k \geq k^*_{T,n}$ we set
\begin{equation}\label{eq-rstar}
r^*_{T,n,k} = \frac{\alpha n^2}{h_{T,n} k^2} \dfrac{1}{1-\dfrac{n^2}{k^2}}. 
\end{equation}
With these preliminaries, the following result holds true. 

\begin{theorem}\label{th-main} 
Let $T$ be such that $T > T^*_M$ for some integer $M \geq 1$. Then, for every integer $n$ with $1 \leq n \leq M$, for every integer
$k \geq k^*_{T,n}$, with $n$ and $k$ relatively prime, and for every $U \in \mathcal{C}_T^{0,1}(\mathbb{R} \times (\mathbb{R}^2 \setminus \{0\}))$, there exists $\varepsilon^* > 0$ such that, for every $\varepsilon \in (0,\varepsilon^*)$, equation \eqref{eq-pert} has at least six $T$-periodic solutions,
denoted by $x_{n,k,+}^{(i)}$ and $x_{n,k,-}^{(i)}$ for $i=1,2,3$, such that:
\begin{itemize}
\item[$(i)$] $|x_{n,k,+}^{(i)}|$ and $|x_{n,k,-}^{(i)}|$ assume the value $r^*_{T,n,k}$ exactly $2n$ times in the interval $[0,T)$, for $i=1,2,3$;
\item[$(ii)$] $x_{n,k,+}^{(i)}$ has winding number $k$ in the interval $[0,T)$ and $x_{n,k,-}^{(i)}$ has winding number $-k$ in the interval $[0,T)$, for $i=1,2,3$.
\end{itemize} 
\end{theorem}

Notice that Theorem~\ref{th-intro} follows by taking $n=1$ in the above statement.

\begin{proof} 
Let us write the perturbed problem \eqref{eq-pert} using the Hamiltonian formalism; according to Section~\ref{subsec-hamilton}, this can be achieved by defining the perturbed Hamiltonian
\begin{equation*}
H_\varepsilon(x,p) = mc^2\sqrt{1+\dfrac{|p|^2}{m^2c^2}}-\dfrac{\alpha}{|x|} - \varepsilon \, U(t,x).
\end{equation*}
Then, we pass to action-angle coordinates, by performing the change of variables
\begin{equation*}
\Sigma \circ \Psi \colon \Lambda_0 \to \mathcal{D} \times \mathbb{T}^2, \qquad (x,p) \mapsto (I_1,I_2,\varphi_1,\varphi_2),
\end{equation*}
described in Section~\ref{section-2.4}. In the new variables, system \eqref{eq-pert} then takes the form
\eqref{hs}, with $N=2$, $I = (I_1,I_2)$, $\varphi = (\varphi_1,\varphi_2)$, $\mathcal{K}(I) = \mathcal{K}_0(I)$ as defined in \eqref{eq-hamactang}
and 
\begin{equation*}
\mathcal{R}(t,I,\varphi) = U(t,x(I,\varphi))
\end{equation*}
where $x(I,\varphi)$ are the first two components of the four-dimensional vector $(\Sigma \circ \Psi)^{-1}(I,\varphi)$.

Our aim is to apply Theorem~\ref{fgg}: we thus need to check that conditions \eqref{toro} and \eqref{hessiano} hold true,
for a suitable $\bar I \in \mathcal{D}$.

To this end, we first observe that, given $n$ and $k$ as in the statement of the theorem, 
the unperturbed problem \eqref{eq-keplero} has periodic solutions with minimal period equal to $T$. Indeed,
let
\begin{equation*}
L_{n,k} = \dfrac{\alpha}{c} \sqrt{\dfrac{1}{1-\dfrac{n^2}{k^2}}}.
\end{equation*}
Then, we easily see that for $h = h_{T,n}$ and $L = L_{n,k}$ conditions \eqref{eq-necperiodiche} and \eqref{eq-commensurabile} are satisfied, so that Proposition~\ref{teo-periodox} applies.
By writing the corresponding action $(\bar I_1, \bar I_2)$ as in \eqref{eq-defi1} and \eqref{eq-defi2}, condition \eqref{toro} plainly follows
(recall that \eqref{toro} means that the corresponding torus is made up by $T$-periodic solutions).

To check \eqref{hessiano}, we compute
\begin{equation*}
\nabla^2 \mathcal{K}_0(I) = 
\dfrac{\alpha^2}{c^2}
\begin{pmatrix}
\dfrac{-3S(I)}{\biggl{(}S(I)^2 + \dfrac{\alpha^2}{c^2}\biggr{)}^{\!\!\frac{5}{2}}} & \dfrac{-3S(I) \dfrac{\partial S(I)}{\partial I_{2}}}{\biggl{(}S(I)^2 + \dfrac{\alpha^2}{c^2}\biggr{)}^{\!\!\frac{5}{2}}} \vspace{7pt} \\
\dfrac{-3S(I)\dfrac{\partial S(I)}{\partial I_{2}}}{\biggl{(}S(I)^2 + \dfrac{\alpha^2}{c^2}\biggr{)}^{\!\!\frac{5}{2}}} & \dfrac{-3S(I)\biggl{(}\dfrac{\partial S(I)}{\partial I_{2}}\biggr{)}^{\!\!^2} + \biggl{(}S(I)^2 + \dfrac{\alpha^2}{c^2}\biggr{)} \dfrac{\partial^{2} S(I)}{\partial I_{2}^{2}} }{\biggl{(}S(I)^2 + \dfrac{\alpha^2}{c^2}\biggr{)}^{\!\!\frac{5}{2}}}
\end{pmatrix},
\end{equation*}
where
\begin{equation*}
S(I) = I_1 - I_2 + \frac{1}{c} \sqrt{c^2 I_2^2 - \alpha^2}.
\end{equation*}
Notice that, by \eqref{eq-defi1} and \eqref{eq-defi2}, $I_1 - I_2 > 0$ and then $S(I) > 0$. As a consequence, 
\begin{equation*}
\det \nabla^2 \mathcal{K}_0(I) = - \frac{\alpha^4}{c^4} \dfrac{3 S(I) \dfrac{\partial^{2} S(I)}{\partial I_{2}^{2}}}{\biggl{(}S(I)^2 +\dfrac{\alpha^2}{c^2}\biggr{)}^{\!\!4}}
= \dfrac{\alpha^6}{c^3}\frac{3 S(I)}{\biggl{(}S(I)^2 +\dfrac{\alpha^2}{c^2}\biggr{)}^{\!\!4} \bigl{(}c^2 I_2^2 - \alpha^2\bigr{)}^{\!\frac{3}{2}}} \neq 0,
\end{equation*}
thus \eqref{hessiano} holds true.

Therefore, Theorem~\ref{fgg} applies so that, for every $\sigma > 0$, three $T$-periodic solutions of \eqref{hs} exist, for $\varepsilon$ small enough.
Going back to the original variables, we have thus found three $T$-periodic solutions of system \eqref{eq-pert}. 
Other three $T$-periodic solutions can be obtained, by repeating the above arguments with $L = - L_{n,k}$. Notice that, 
up to taking smaller values of $\sigma$ (and thus of $\varepsilon$) all the six solutions remain distinct, since 
\eqref{vicinotoro} ensures that they stay close to solutions on distinct tori of the unperturbed problem. To conclude the proof, we thus need to show that, again up to taking smaller values of $\sigma$, conditions $(i)$ and $(ii)$ of the statement are satisfied.

As for $(ii)$, we notice that from \eqref{vicinotoro} it follows that
\begin{equation}\label{vicino1}
\max_{t \in [0,T]} | x(t) - x_0(t) | + \max_{t \in [0,T]} |p(t) - p_0(t) | \leq C(\sigma),
\end{equation}
with $(x_0,p_0)$ a $T$-periodic solution of \eqref{eq-keplero} with $h = h_{T,n}$ and $L = \pm L_{n,k}$, and $C(\sigma)$ a suitable constant with $C(\sigma) \to 0$ as $\sigma \to 0^+$. 
Since, by Proposition~\ref{teo-periodox}, $x_0$ has winding number $\pm k$ on the interval $[0,T)$ according to $L = \pm L_{n,k}$,
and the winding number is continuous in the $\mathcal{C}^0$-topology, condition $(ii)$ follows.

To check that $(i)$ holds as well, we first claim that $x$ remains close to $x_0$ also in the $\mathcal{C}^1$-norm.
To see this, we use \eqref{eq-hs} and its perturbed version to infer that
\begin{equation}\label{vicino2}
\dot x(t) - \dot x_0(t) = \dfrac{p(t)}{m \sqrt{1 + \dfrac{| p(t) |^2}{m^2 c^2}}} - \frac{p_0(t)}{m \sqrt{1 + \dfrac{| p_0(t) |^2}{m^2 c^2}}}, \quad \text{for every $t \in [0,T]$.}
\end{equation}
An easy computation shows that the Jacobian matrix of the function
\begin{equation*}
\mathbb{R}^2 \ni p \mapsto \dfrac{p}{m \sqrt{1 + \dfrac{| p|^2}{m^2 c^2}}}
\end{equation*}
is globally bounded; then, using the mean-value theorem, from \eqref{vicino2} we infer that 
\begin{equation*}
| \dot x(t) - \dot x_0(t) | \leq \Gamma \, | p(t) - p_0(t) |, \quad \text{for every $t\in[0,T]$,}
\end{equation*}
for a suitable constant $\Gamma > 0$, and thus, by \eqref{vicino1},
\begin{equation*}
\max_{t \in [0,T]} | \dot x(t) - \dot x_0(t) | \leq \Gamma \, C(\sigma).
\end{equation*}
Writing $x = r e^{i\vartheta}$ as in \eqref{eq-cambiovar2}, we have
\begin{equation*}
\dot r = \frac{\langle x,\dot x \rangle}{r}
\end{equation*}
and we can thus deduce that $r$ is close to $r_0$ in the $\mathcal{C}^1$-topology. 
Once this is known, we can conclude by showing that $r_0$ assumes exactly $2n$ times in the interval $[0,T)$ the value
$r^*_{T,n,k}$, always with non-zero derivative.
This follows from $(i)$ of Proposition~\ref{teo-periodox}, after checking that the definition of $r^*$ in \eqref{eq-rstarbasso} for $h = h_{T,n}$
and $L = L_{n,k}$ coincides with the one of the value $r^*_{T,n,k}$ given in \eqref{eq-rstar}.
\end{proof}

\begin{remark}[Abundance of $T$-periodic solutions]
According to \cite[Corollary~2.2]{FoGaGi-16}, it would be possible to bifurcate also from tori filled by periodic solutions with minimal period $T/\ell$, for some integer $\ell \geq 2$. As a consequence, on growing of the period $T$, even more $T$-periodic solutions to equation \eqref{eq-pert} could be provided, leading to an improved version of Theorem~\ref{th-intro}. For briefness, we prefer not to go into the details of this. 
\end{remark}

\begin{remark}[Applying KAM theory]\label{rem-kam}
Since the non-degeneracy condition \eqref{hessiano} has been proved to hold
for the unperturbed problem \eqref{eq-keplero} when written in action-angle coordinates, a standard KAM theorem (see, for instance, \cite[Chapter~4.5]{Du-14}) can be applied
to ensure the survival of quasi-periodic invariant tori for a perturbed problem of the type
\begin{equation*}
\dfrac{\mathrm{d}}{\mathrm{d}t}\left(\dfrac{m\dot{x}}{\sqrt{1-|\dot{x}|^2/c^2}}\right)=-\alpha\, \dfrac{x}{|x|^3}+\varepsilon \, \nabla U(x),
\end{equation*}
with $U$ sufficiently regular and $\varepsilon$ sufficiently small. We stress that here the perturbation term
$U$ is required to depend only on $x$, and not on time $t$.
We are not aware of KAM theorems allowing to deal with a time-periodic perturbation of the relativistic Kepler problem, as in \eqref{eq-pert}  
(compare for instance with \cite[Theorem~1]{GuKa-PP}, requiring however the Hessian of the Hamiltonian to be strictly convex, which is not the case in our setting).
\end{remark}

\bibliographystyle{elsart-num-sort}
\bibliography{BoDaFe-biblio}

\begin{thebibliography}{10}
\expandafter\ifx\csname url\endcsname\relax
  \def\url#1{\texttt{#1}}\fi
\expandafter\ifx\csname urlprefix\endcsname\relax\def\urlprefix{URL }\fi

\bibitem{AmCo-89}
A.~Ambrosetti, V.~Coti~Zelati, Perturbation of {H}amiltonian systems with
  {K}eplerian potentials, Math. Z. 201 (1989) 227--242.

\bibitem{AmCoEk-87}
A.~Ambrosetti, V.~Coti~Zelati, I.~Ekeland, Symmetry breaking in {H}amiltonian
  systems, J. Differential Equations 67 (1987) 165--184.

\bibitem{AnBa-71}
C.~M. Andersen, H.~C. von Baeyer, On classical scalar field theories and the
  relativistic {K}epler problem, Ann. Physics 62 (1971) 120--134.

\bibitem{Ar-89}
V.~I. Arnol'd, Mathematical methods of classical mechanics, vol.~60 of Graduate
  Texts in Mathematics, 2nd ed., Springer-Verlag, New York, 1989.

\bibitem{BeKa-87}
D.~Bernstein, A.~Katok, Birkhoff periodic orbits for small perturbations of
  completely integrable {H}amiltonian systems with convex {H}amiltonians,
  Invent. Math. 88 (1987) 225--241.

\bibitem{BoDaPa-20}
A.~Boscaggin, W.~Dambrosio, D.~Papini, Periodic solutions to a forced {K}epler
  problem in the plane, Proc. Amer. Math. Soc. 148 (2020) 301--314.

\bibitem{BoOrZh-19}
A.~Boscaggin, R.~Ortega, L.~Zhao, Periodic solutions and regularization of a
  {K}epler problem with time-dependent perturbation, Trans. Amer. Math. Soc.
  372 (2019) 677--703.

\bibitem{Bo-04}
T.~H. Boyer, Unfamiliar trajectories for a relativistic particle in a {K}epler
  or {C}oulomb potential, Amer. J. Phys. 72 (2004) 992--997.

\bibitem{CaVi-00}
H.~Cabral, C.~Vidal, Periodic solutions of symmetric perturbations of the
  {K}epler problem, J. Differential Equations 163 (2000) 76--88.

\bibitem{Ch-92}
W.~F. Chen, Birkhoff periodic orbits for small perturbations of completely
  integrable {H}amiltonian systems with nondegenerate {H}essian, in: Twist
  mappings and their applications, vol.~44 of IMA Vol. Math. Appl., Springer,
  New York, 1992, pp. 87--94.

\bibitem{Du-14}
H.~S. Dumas, The {KAM} story. {A} friendly introduction to the content,
  history, and significance of classical {K}olmogorov-{A}rnold-{M}oser theory,
  World Scientific Publishing Co. Pte. Ltd., Hackensack, NJ, 2014.

\bibitem{FoGa-17}
A.~Fonda, A.~C. Gallo, Radial periodic perturbations of the {K}epler problem,
  Celestial Mech. Dynam. Astronom. 129 (2017) 257--268.

\bibitem{FoGa-18}
A.~Fonda, A.~C. Gallo, Periodic perturbations with rotational symmetry of
  planar systems driven by a central force, J. Differential Equations 264
  (2018) 7055--7068.

\bibitem{FoGaGi-16}
A.~Fonda, M.~Garrione, P.~Gidoni, Periodic perturbations of {H}amiltonian
  systems, Adv. Nonlinear Anal. 5 (2016) 367--382.

\bibitem{FoUr-17}
A.~Fonda, A.~J. Ure\~{n}a, A higher dimensional {P}oincar\'{e}-{B}irkhoff
  theorem for {H}amiltonian flows, Ann. Inst. H. Poincar\'{e} Anal. Non
  Lin\'{e}aire 34 (2017) 679--698.

\bibitem{Ga-19}
A.~C. Gallo, Periodic solutions of perturbed central {H}amiltonian systems,
  NoDEA Nonlinear Differential Equations Appl. 26 (2019) Art. 34, 24 pp.

\bibitem{GuKa-PP}
M.~Guardia, V.~Kaloshin, Orbits of nearly integrable systems accumulating to
  {KAM} tori, arXiv:1412.7088.

\bibitem{MoZe-05}
J.~Moser, E.~J. Zehnder, Notes on dynamical systems, vol.~12 of Courant Lecture
  Notes in Mathematics, New York University, Courant Institute of Mathematical
  Sciences, New York; American Mathematical Society, Providence, RI, 2005.

\bibitem{MuPa-06}
G.~Mu\~{n}oz, I.~Pavic, A {H}amilton-like vector for the special-relativistic
  {C}oulomb problem, European J. Phys. 27 (2006) 1007--1018.

\bibitem{Po-76}
H.~Pollard, Celestial mechanics, vol.~18 of Carus Mathematical Monographs,
  Mathematical Association of America, Washington, DC, 1976.

\bibitem{SeTe-94}
E.~Serra, S.~Terracini, Noncollision solutions to some singular minimization
  problems with {K}eplerian-like potentials, Nonlinear Anal. 22 (1994) 45--62.

\bibitem{ToUrZa-13}
P.~J. Torres, A.~J. Ure\~{n}a, M.~Zamora, Periodic and quasi-periodic motions
  of a relativistic particle under a central force field, Bull. Lond. Math.
  Soc. 45 (2013) 140--152.

\bibitem{Za-13}
M.~Zamora, New periodic and quasi-periodic motions of a relativistic particle
  under a planar central force field with applications to scalar boundary
  periodic problems, Electron. J. Qual. Theory Differ. Equ. (2013) No. 31, 16
  pp.

\end{thebibliography}

\end{document}